\documentclass[12pt]{amsart}
\usepackage{amssymb}
\newcommand{\D}{\mathcal{D}}

\newtheorem{theorem}{Theorem}[section]
\newtheorem*{theoremintro}{Main Theorem}

\newtheorem{lemma}[theorem]{Lemma}
\newtheorem{corollary}[theorem]{Corollary}
\theoremstyle{remark}
\newtheorem{remark}[theorem]{Remark}

\theoremstyle{definition}
\newtheorem{definition}[theorem]{Definition}

\newcommand{\dyadicSquares}{\mathcal{D}^d}
\newcommand{\dyadicInterval}{\mathcal{D}}
\newcommand{\mathI}{\mathcal{I}}
\newcommand{\mathJ}{\mathcal{J}}
\newcommand{\mathK}{\mathcal{K}}
\newcommand{\mathM}{\mathcal{M}}

\newcommand{\mathP}{\mathcal{P}}
\newcommand{\mathR}{\mathcal{R}}
\newcommand{\mathS}{\mathcal{S}}

\newcommand{\sizeOfSet}[1]{\mid #1\mid}
\newcommand{\coefAt}[3]{c_{#1}^{(#2)}(#3)}
\newcommand{\coefAbs}[3]{\mid c_{#1}^{(#2)}(#3)\mid}
\newcommand{\coefAbsGen}[2]{\mid c_{#1}^{(#2)}\mid}
\newcommand{\haar}[2]{h_{#1}^{(#2)}}
\newcommand{\Id}[1]{\chi_{#1}}

\newcommand{\chain}[2]{\mathbb{C}(#1,#2)}
\newcommand{\father}[1]{F(#1)}
\newcommand{\spectrum}[1]{sp(#1)}
\newcommand{\SPectrum}[1]{SP(#1)}
\newcommand{\Ld}{L_1[0,1]^d}
\newcommand{\LtwoDim}{L_1[0,1]^2}
\newcommand{\son}[2]{\mathit{son}(#1,#2)}
\newcommand{\sonOrder}[3]{\mathit{son}^{#3}(#1,#2)}

\begin{document}

\baselineskip 16.5pt
\title[Convergence of Weak Greedy algorithm] {On the Convergence of a Weak Greedy Algorithm for the Multivariate Haar Basis}
\author[Dilworth]{S. J. Dilworth}
\address{Department of Mathematics\\ University of South Carolina\\
Columbia, SC 29208\\ U.S.A.} \email{dilworth@math.sc.edu}
\author[Gogyan]{S. Gogyan}
\address{Institute of Mathematics,\\ Armenian Academy of Sciences, 24b Marshal Baghramian ave.\\ 
Yerevan, 0019\\ Armenia }
\email{gogyan@instmath.sci.am}
\author[Kutzarova]{Denka Kutzarova}
\address{Institute of Mathematics, Bulgarian Academy of
Sciences, Sofia, Bulgaria.} \curraddr{Department of Mathematics, University of Illinois
at Urbana-Champaign, Urbana, IL 61801, U.S.A.} \email{denka@math.uiuc.edu}
\thanks{The first author was partially supported by NSF grant DMS1101490.  
All authors were participants at the NSF supported
Workshop in Analysis and Probability at Texas A\&M University in 2011.}

\begin{abstract} We define a family of weak thresholding greedy algorithms for the multivariate Haar basis for $\Ld$ ($d \ge 1$).
We prove convergence and uniform boundedness of the weak greedy approximants for all $f \in \Ld$.
\end{abstract} \maketitle \tableofcontents
\section{Introduction}\label{sec: intro}

Let $\Psi= (\psi_n)_{n=1}^\infty$ be a semi-normalized Schauder basis for a Banach space $X$. For $f \in X$, let $(c_n(f))_{n=1}^\infty$ denote the sequence of basis coefficients 
for $f$. The \textit{Thresholding Greedy Algorithm} (TGA) was introduced by Temlyakov \cite{T3} for the trigonometric system and subsequently extended to the Banach space setting by Konyagin and Temlyakov \cite{KT1}. See \cite{T1} and the recent monograph \cite{T2} for the history of the problem and for background information on greedy approximation. The algorithm is defined as follows. For $f \in X$ and $n \ge1$, let $\Lambda_n(x) \subset \mathbb{N}$ be the indices corresponding to a choice of $n$ largest coefficients of $f$ in absolute value, i.e., $\Lambda_n(f)$ satisfies
\begin{equation} \label{eq: tga} \min\{|c_i(f)| \colon i \in \Lambda_n(f)\} \ge \max\{|c_i(f)| \colon i \notin \Lambda_n(f)\}. \end{equation}
(Note that $\Lambda_n(f)$ is uniquely defined if and only if there is strict inequality in \eqref{eq: tga}.)  
We call $G_n(f) := \sum_{i \in \Lambda_n(f)} c_i(f) \psi_i$  an \textit{$n^{th}$ greedy approximant to $f$} and say that the TGA converges if $G_n(f) \rightarrow f$. The basis $\Psi$ is said to be \textit{quasi-greedy} if there exists $K < \infty$ such that for all $f \in X$  and $n \ge 1$, we have $\|G_n(f)\|\le K \|f\|$. Wojtaszczyk \cite[Theorem~1]{W} proved that $\Psi$ is quasi-greedy if and only if the TGA converges for all initial vectors $f \in X$.

It was proved in \cite[Remark~6.3]{DKW} that the one-dimensional Haar basis for $L_1[0,1]$ (normalized in $L_1[0,1]$) is \textit{not} quasi-greedy, i.e., that the TGA does not converge for certain initial vectors $f$. However, it was proved in \cite{G} that there is a \textit{weak thresholding greedy algorithm} (WTGA) for the Haar basis which  converges. 

A  WTGA is a procedure of  the following general type.  Fix a \textit{weakness parameter} $t$ with $0 < t < 1$. For each $f \in X$, define an increasing sequence $(\Lambda^t_n(f))$ of sets, consisting of  of $n$ coefficient indices, such that
\begin{equation} \label{eq: wtga} \min\{|c_i(f)| \colon i \in \Lambda^t_n(f)\} \ge t\max\{|c_i(f)| \colon i \notin \Lambda_n(f)\}. \end{equation}
The  WTGA is said to converge if the sequence of  \textit{weak greedy approximants} $G^t_n(f) := \sum_{i \in \Lambda^t_n(f)} c_i(f) \psi_i$ converges  to $f$.  It was proved in \cite{KT2} that quasi-greediness of $\Psi$ guarantees  convergence for every WTGA. 

However, if $\Psi$ is \textit{not} quasi-greedy then the index sets $(\Lambda^t_n(f))$ must be carefully chosen to ensure convergence. For the WTGA defined in \cite{G} it was proved that the algorithm converges and that the weak greedy approximants are uniformly bounded, i.e., that
$\|G^t_n(f)\| \le K(t)\|f\|$, where $K(t)$ depends only on the weakness parameter.

In \cite{T} it was proved for the multivariate Haar system, normalized in $L_p[0,1]^d$, for $d \ge 1$ and $1 < p < \infty$, that $$\|f - G_n(f)\| \le C(p,d) \sigma_n(f),$$
where $\sigma_n(f)$ denotes the error in the best $n$-term approximation to $f$ in the  $L_p$ norm using the multivariate Haar system. As remarked above, convergence fails for $p=1$.
The goal of the present paper is to extend the results of \cite{G} to the multivariate Haar system for $L_1[0,1]^d$.  The case $d =2$ is especially interesting from the point of view of practical applications, and we refer the reader to \cite{BF} (and its references)  for a nice exposition of the two-dimensional discrete Haar wavelet transform and its use in image compression. 

Some serious obstacles have to be overcome in extending the one-dimensional results  to higher dimensions. These difficulties  are for the most part already present in the case $d=2$. On the other hand, the passage  from $d=2$  to $d \ge 3$
is relatively straightforward. 

The first obstacle in extending the one-dimensional algorithm of \cite{G}, which impeded progress on this problem for a considerable period, is that  `` the obvious generalization''   fails to converge. Therefore, a more complicated algorithm is required, which depends on \textit{two} parameters: the weakness parameter $t$ and  a second parameter $s$, where $0 < t < s$. An important feature of the algorithm, which it shares with the simpler one-dimensional algorithm,  is that  the weak greedy approximant is updated by applying a basic \textit{greedy step} to the residual vector $R^t_n(f)=f - G^t_n(f)$. The form of this  greedy step  ensures that the algorithm is \textit{branch-greedy} in the sense of \cite{DKSW}. Roughly speaking, this means that the selection of the next coefficient 
in the basic greedy step \textit{depends only on} the natural (finite)  data set for weak thresholding consisting of all pairs
$$\{(i, c_i(R_n(f))) \colon |c_i(R_n(f)| \ge t \max_{j \ge 1} |c_j(R_n(f)|\}.$$

Let us now describe the contents of the paper.  In Section~2 we recall the definition of the multivariate Haar system, describe the weak threshoding greedy algorithm alluded to in the title of the paper, and state the Main Theorem.  The proof of the Main Theorem, which is presented in Section~7, uses two key lemmas which are proved in Sections~3--6. 

 The main result of Section~3 is the norm estimate Lemma~\ref{lem: coefDiff_Neighboor}. The results in Section~4 are based on the combinatorics of  dyadic cubes. The main result of this section is the first key lemma, namely the norm estimate Lemma~\ref{lem: from_q_to_qtilda}. Sections~5 and 6 are independent of Section~4. Section~5 contains an important
symmetrization result. Section~6 is devoted to the second key lemma. The proof of this lemma uses an induction argument which makes essential use of the symmetrization results of the previous section. 

 In Section~8 we show that the algorithms diverge for the boundary cases $s=t$ and $s=1$. This implies, in particular, that the multivariate Haar system is not quasi-greedy (Corollary~\ref{cor: notQG}).

\section{Multivariate Haar system, Definition of the Algorithm, and Main Theorem} \
In this paper we consider greedy algorithms for the multivariate Haar system. Let $\dyadicInterval_n$ be the set of dyadic intervals of length $2^{-n}$ and  let $\dyadicSquares_n := \dyadicInterval_n\times\ldots\times\dyadicInterval_n$ be the collection of all $d$-dimensional  dyadic cubes of side length $2^{-n}$.  Further, let $\dyadicInterval := \bigcup_{n\geq 1}\dyadicInterval_n$ and $\dyadicSquares := \bigcup_{n\geq 1}\dyadicSquares_n$.

For $a<b$, let
\begin{equation*}
r_{[a,b)}^{(0)} := {\Id{[a,b)}\over b-a},\quad r_{[a,b)}^{(1)}= {\Id{[a, {a+b\over 2})}-\Id{[{a+b\over 2},b)}\over b-a},
\end{equation*}
where $\chi$ denotes the characteristic function. 

There are $2^d-1$ different Haar functions corresponding to the dyadic cube $\mathI\in\dyadicSquares$, namely, for $1\leq j\leq 2^d-1$, 
\begin{equation*}
h_\mathI^{(j)}(x) := \prod_{k=1}^d r_{\mathI_k}^{(\epsilon_k)}(x_k),
\end{equation*}
where  $x=(x_1,\ldots,x_d)\in [0,1)^d$, $\mathI=\mathI_1\times\ldots\times\mathI_d$ and $\epsilon_k \in \{0,1\}$ are defined from the binary representation $j=\sum_{k=1}^d \epsilon_k 2^{d-k}$. The Haar system is the set of all functions $h_\mathI^{(j)}$ together with $\Id{[0,1)^d}$. The Haar coefficients are defined by
\begin{equation} \label{eq: haarcoef}
\coefAt{\mathI}{i}{f} = \mu(\mathI) \int_\mathI f h_\mathI^{(i)},
\end{equation}
where the integral is taken with respect to  ($d$-dimensional) Lebesgue measure $\mu$.
We write $[a,b)\prec[c,d)$ if $b-a > d-c$ or $b-a=d-c$ and $a<c$. Further, for $\mathI=\mathI_1\times\ldots\times\mathI_d\in\dyadicSquares$, $\mathJ=\mathJ_1\times\ldots\times\mathJ_d\in\dyadicSquares$, we write $\mathI\prec\mathJ$ if $\mathI$ precedes $\mathJ$ in the lexicographic ordering, i.e.,
\begin{itemize}
\item\ $\mathI_1\prec\mathJ_1$, or
\item\ $\mathI_1 = \mathJ_1$ and $\mathI_2 \prec \mathJ_2$, or \\
...............
\item\ $\mathI_1 = \mathJ_1,\ldots, \mathI_{d-1} = \mathJ_{d-1}$ and $\mathI_d \prec \mathJ_d$.
\end{itemize} 

Finally, we write $(\mathI,i)\prec(\mathJ,j)$ if $\mathI\prec\mathJ$ or $\mathI=\mathJ$ and $i<j$.

Note that each $\mathI \in \dyadicSquares_{n}$ is the disjoint union of $2^d$ subcubes belonging to $\dyadicSquares_{n+1}$ which we shall refer to as the  \textit{immediate successors} of $\mathI$.

Now we are ready to define the algorithm. For any $f\in\Ld$ and $0<t\leq s\leq 1$  we define the sequence $\{G_m^{t,s}(f)\}$ inductively. We put $G_0^{s,t}=G_0^{s,t}(f)=0$,\ $R_0^{s,t}=R_0^{s,t}(f)=f$ and for each $m\geq 1$ we define $G_m^{s,t}(f)$ and $R_m^{s,t}(f)$ in the following way

\begin{itemize}
\item[1)] Find the first cube in the order $\prec$, denoted $\Delta_m$, and then the smallest value of $j_m$, with $1 \le j_m \le 2^{d}-1$, such that $(\Delta_m,j_m)$ satisfies
\begin{equation*}
\mid\coefAt{\Delta_m}{j_m}{R_{m-1}^{s,t}}\mid = \max_{\mathI\in\dyadicSquares}\mid \coefAt{\mathI}{i}{R_{m-1}^{s,t}}\mid
\end{equation*}
\item[2)] Define $\tilde{\Delta}_m\supseteq\Delta_m$ to be the largest cube containing $\Delta_m$ such that for every $\mathI \in \dyadicSquares$ with 
$\Delta_m\subseteq\mathI\subseteq\tilde{\Delta}_m$ there exists $1\leq i\leq 2^d-1$ such that
\begin{equation*}
\mid\coefAt{\mathI}{i}{R_{m-1}^{s,t}}\mid \geq s \mid \coefAt{\Delta_m}{j_m}{R_{m-1}^{s,t}}\mid.
\end{equation*}
\item[3)] Find $1\leq i_m\leq 2^d-1$ for which $\mid \coefAt{\tilde{\Delta}_m}{i_m}{R_{m-1}^{s,t}}\mid$ is the smallest value to satisfy
\begin{equation} \label{eq: conditionstep3}
\mid\coefAt{\tilde{\Delta}_m}{i_m}{R_{m-1}^{s,t}}\mid \geq {t\over s} \max_{1\leq j\leq 2^d-1}\mid \coefAt{\tilde{\Delta}_m}{j}{R_{m-1}^{s,t}}\mid.
\end{equation}
\item[4)] Let
\begin{equation*}
G_m^{s,t}(f)  := G_{m-1}^{s,t} + \coefAt{\tilde{\Delta}_m}{i_m}{f}\haar{\tilde{\Delta}_m}{i_m},\  R_m^{s,t} := f - G_m^{s,t}
\end{equation*} be the updated greedy approximant and residual.

\end{itemize}

\begin{remark}
 Step 3)  may be modified by replacing the selection condition \eqref{eq: conditionstep3} by the following condition:
\begin{equation}
\mid\coefAt{\tilde{\Delta}_m}{i_m}{R_{m-1}^{s,t}}\mid \geq t \mid \coefAt{\Delta_m}{j_m}{R_{m-1}^{s,t}}\mid.
\end{equation}
This defines  a second weak greedy algorithm which has exactly the same convergence properties as the first algorithm.
\end{remark}

Now we can state our main result.

\begin{theoremintro} Let $0 < t < s < 1$. Then the weak greedy algorithms defined above converge, i.e.,  $f = \lim_{m \rightarrow\infty} G_m^{s,t}(f)$ for every $f \in \Ld$. Moreover, the greedy approximants are uniformly bounded, i.e., for all $f \in \Ld$ and for all $m \ge 1$, 
\begin{equation}  \label{eq: uniformbounded}\|G_m^{s,t}(f)\| \le C(d,s,t) \|f\|. \end{equation}
\end{theoremintro}

In Section~7 we show that $G_m^{s,t}$ does not converge when $s=t$ or $s=1$. From this, in particular from the case $s=t=1$, it follows that the multivariate Haar system is not a quasi-greedy basis in $\Ld$. 

\section{Norm Estimates by Expansion Coefficients}\label{sec: coefficients}
For any $f\in\Ld$, let
\begin{equation}\label{def: spectrum} 
\spectrum{f} := \{(\Delta,i)\ :\ \coefAt{\Delta}{i}{f}\neq 0\},
\end{equation}
and let
\begin{equation}\label{def: SPectrum} 
\SPectrum{f} := \{\Delta\ :\ (\Delta,i)\in\spectrum{f}\ \hbox{for some}\ 1\leq i\leq 2^d-1\}.
\end{equation}

Also, for any $\mathI\in\dyadicSquares$, let
\begin{equation}\label{restriction1} 
\mathP_\mathI f := f - \sum_{\Delta\subseteq\mathI}\sum_{i=1}^{2^d-1}\coefAt{\Delta}{i}{f} \haar{\Delta}{i}.
\end{equation}

Note that $\mathP_\mathI f$ is constant on $\mathI$ and it coincides with $f$ outside of $\mathI$. Recall, that the Haar system has the following 
monotonicity property:

\begin{equation*}\label{restriction} 
\Vert\mathP_\mathI f\Vert \geq \Vert\mathP_\mathJ f\Vert,\ \hbox{ for any }\ \mathI\subset\mathJ.
\end{equation*}
Also, let us denote the norm of $f\in\Ld$ on the set $\Delta$ by $\Vert f\Vert_\Delta$, i.e.,
\begin{equation*} \Vert f\Vert_\Delta := \int_\Delta |f|. \end{equation*}

Now we formulate two basic lemmas.

\begin{lemma} \label{lem: estimate_byCoef} Let $f\in\Ld$ and $\mathI, \mathJ\in\dyadicSquares$ and $\mathJ\subseteq\mathI$. Then
\begin{equation*}
\Vert f\Vert_\mathI\geq \coefAbs{\mathJ}{i}{f}\ \hbox{for any }\ 1\leq i\leq 2^d-1.
\end{equation*}
\end{lemma}
\begin{proof} By \eqref{eq: haarcoef},
$$ \coefAbs{\mathJ}{i}{f}= \mu(\mathJ)\left| \int_\mathJ f h_\mathJ^{(i)}\right| \le \int_\mathJ |f| = \|f\|_\mathJ \le \|f\|_\mathI.$$
\end{proof}
\begin{lemma} \label{lem: estimate_forQ} Let $f\in\Ld$ and suppose $\coefAbs{\Delta}{i}{f}\leq 1$ for all $\Delta\in\dyadicSquares$ and for all $1\leq i\leq 2^d-1$. Then, for any dyadic cube $\mathI\in\dyadicSquares$, one has
\begin{equation*}
\Vert\mathP_\mathI f\Vert_\mathI\leq 1.
\end{equation*}
\end{lemma}
\begin{proof} Let $k$ be defined by  $\mu(\mathI):=2^{-kd}$. Let the chain of dyadic cubes  $\{\Delta_j\}_{j=0}^{k-1}$ be defined by the following conditions:
\begin{equation*}
i)\ \mathI\subset\Delta_j,\quad ii)\ \mu(\Delta_j)=2^{-jd}.
\end{equation*}
It is clear that $\mathP_\mathI f$ is constant on $\mathI$ and that for any $x\in\mathI$ one has
\begin{equation*}
\mathP_\mathI f(x) = \coefAt{\Delta_0}{0}{f}  + \sum_{j=0}^{k-1}\sum_{i=1}^{2^d-1} \coefAt{\Delta_j}{i}{f} \haar{\Delta_j}{i}(x).
\end{equation*}
Taking into account the fact that $\mid\haar{\Delta_j}{i}(x)\mid=2^{jd}$, we conclude that
\begin{equation*}
\Vert\mathP_\mathI(f)\Vert_\mathI \leq 2^{-kd}\Biggl(1 + \sum_{j=0}^{k-1}(2^d-1)\cdot 2^{jd}\Biggr)=1.
\end{equation*}
\end{proof}

\begin{lemma}\label{lem: coefDiff_SameLevel} Suppose $\mathJ\subset\mathI$ and $\mu(\mathJ)={\mu(\mathI)\over 2^d}$, where $\mathI, \mathJ \in \dyadicSquares$. Then
\begin{equation*}
\Vert f\Vert_{\mathI\setminus\mathJ} \geq \begin{cases} \mid \Vert\mathP_\mathI(f)\Vert_\mathI - \Vert\mathP_\mathJ(f)\Vert_\mathJ\mid\\
 {1\over 4}\mid \coefAbsGen{\mathI}{i}-\coefAbsGen{\mathI}{j}\mid\\ 
{1\over 4}\mid \Vert\mathP_\mathI(f)\Vert_\mathI-\coefAbsGen{\mathI}{j}\mid.  \end{cases} 
\end{equation*} For $d=2$ the constant $1/4$ may be improved to $1/2$.
\begin{proof} Here we  prove the lemma (with the improved constant of $1/2$) only for the case $d=2$. The case $d\geq 3$ is proved in  Section~9.

Let us prove the first statement. The first inequality in \eqref{eq: firststatement} below  follows from the monotonicity property of the Haar system, while the second  follows from the fact that $\mathP_\mathI(f)$ and $\mathP_\mathJ(f)$ are constant on $\mathI$ and $\mathJ$ respectively:
\begin{equation} \label{eq: firststatement} \begin{split}
\Vert f\Vert_{\mathI\setminus\mathJ} &\geq \mid \int_\mathI f - \int_\mathJ f\mid\\
&= | \int_\mathI \mathP_\mathI(f) - \int_\mathJ \mathP_\mathJ(f)| \\ 
&\geq \mid \Vert\mathP_\mathI(f)\Vert_\mathI - \Vert\mathP_\mathJ(f)\Vert_\mathJ\mid.
\end{split} \end{equation}

Let us prove the second and third statements of the lemma.
Let $\delta := \mu(\mathI)$. Let us  denote the value of $\mathP_\mathI(f)$ on the cube $\mathI$ by $H$ and the value of $\mathP_\mathJ(f)$ on the cube $\mathJ$ by $H_1$.

Clearly,
\begin{equation}\label{Norm_on_I}
\Vert\mathP_\mathI(f)\Vert_\mathI = \delta\mid H\mid,
\end{equation}
and, since $d=2$,
\begin{equation}\label{Norm_on_J}
\Vert\mathP_\mathJ(f)\Vert_\mathJ = {\delta\mid H_1\mid\over 4}.
\end{equation}

Note that for some choice of signs $\epsilon_1,\epsilon_2, \epsilon_3 =\pm 1$ one has
\begin{equation*}
H_1 = H + {\epsilon_1 c_{\mathI}^{(1)} + \epsilon_2 c_{\mathI}^{(2)} + \epsilon_3 c_{\mathI}^{(3)} \over \delta}.
\end{equation*}
Let $a_i := \epsilon_i c_{\mathI}^{(i)}$. Then by the monotonicity property of the Haar system one has

\begin{equation*}
\Vert f\Vert_{\mathI\setminus\mathJ} \geq {1\over 4} \Biggl(\mid H\delta+a_1-a_2-a_3\mid + \mid H\delta-a_1+a_2-a_3\mid + \mid H\delta-a_1-a_2+a_3\mid \Biggr).
\end{equation*}
Combining this inequality with \eqref{Norm_on_I} and using the triangle inequality one obtains the second and third  statements of the lemma with $1/4$ replaced by $1/2$.
\end{proof}
\end{lemma}

\begin{lemma}\label{lem: coefDiff_Neighboor} 
Let $\mathK\subset\mathJ\subset\mathI$ with $\mu(\mathJ)={\mu(\mathI)\over 2^d}$ and $\mu(\mathK)={\mu(\mathJ)\over 2^d}$. Then
\begin{equation*}								
\Vert f\Vert_{\mathI\setminus\mathK} \geq \mid { \coefAbsGen{\mathI}{i}- \coefAbsGen{\mathJ}{j}\over 8} \mid,
\end{equation*}
for all $1\leq i,j\leq 2^d-1$.
\end{lemma}
\begin{proof}
Using Lemma \ref{lem: coefDiff_SameLevel} one gets
\begin{equation*}\Vert f\Vert_{\mathI\setminus\mathJ} \ge \frac{1}{2}\left(\mid{ \coefAbsGen{\mathI}{i}- \Vert\mathP_\mathI(f)\Vert_\mathI\over 4}\mid  + \mid{ \Vert\mathP_\mathI(f)\Vert_\mathI - \Vert\mathP_\mathJ(f)\Vert_\mathJ}\mid\right) \end{equation*}
and 
\begin{equation*}\Vert f\Vert_{\mathJ\setminus\mathK} \ge \mid{  \coefAbsGen{\mathJ}{j}- \Vert\mathP_\mathJ(f)\Vert_\mathJ\over 4}\mid. \end{equation*}
Hence
\begin{equation*}
\begin{split}
\Vert f\Vert_{\mathI\setminus\mathK} =& \Vert f\Vert_{\mathI\setminus\mathJ} + \Vert f\Vert_{\mathJ\setminus\mathK} \\
\geq & \mid{ \coefAbsGen{\mathI}{i}- \Vert\mathP_\mathI(f)\Vert_\mathI\over 8}\mid  + \mid{ \Vert\mathP_\mathI(f)\Vert_\mathI - \Vert\mathP_\mathJ(f)\Vert_\mathJ\over 2}\mid\\
+ & \mid{  \coefAbsGen{\mathJ}{j}- \Vert\mathP_\mathJ(f)\Vert_\mathJ\over 4}\mid\\
\geq & \mid { \coefAbsGen{\mathI}{i}- \coefAbsGen{\mathJ}{j}\over 8}\mid. 
\end{split}
\end{equation*}
\end{proof}

\begin{lemma}\label{lem: coefDiff_general}
Let $\mathK\subsetneq\mathJ\subset\mathI$. Then
\begin{equation*}								
\Vert f\Vert_{\mathI\setminus\mathK} \geq \mid { \coefAbsGen{\mathI}{i}- \coefAbsGen{\mathJ}{j}\over 16} \mid.
\end{equation*}
\end{lemma}
\begin{proof}
If $\mathI = \mathJ$ then the the lemma follows from  Lemma \ref{lem: coefDiff_SameLevel}. Suppose $\mathI\neq\mathJ$. Put $\Delta_0 = \mathI$ and inductively define the chain $\{\Delta_k\}_{k=0}^m$ of dyadic cubes as follows:
\begin{equation*}
\mathK\subset\Delta_{k+1}\subset\Delta_{k},\ \mu(\Delta_{k+1})={\mu(\Delta_{k})\over 2^d} \qquad (0 \le k <m ).
\end{equation*}
Note that $\Delta_m = \mathK$  and that   $\mathJ = \Delta_p$ for some $0<p<s$. So by  Lemma  \ref{lem: coefDiff_Neighboor} we get
\begin{equation*}
\begin{split}
&\Vert f\Vert_{\mathI\setminus\mathK}\geq\Vert f\Vert_{\Delta_0\setminus\Delta_{p+1}} \geq {1\over 2}\sum_{k=0}^{p-1}\Vert f\Vert_{\Delta_{k}\setminus\Delta_{k+2}}\\
&\geq \mid {\coefAbsGen{\mathI}{i} - \coefAbsGen{\Delta_1}{1}\over 16}\mid + \sum_{k=1}^{p-2} \mid {\coefAbsGen{\Delta_k}{1} - \coefAbsGen{\Delta_{k+1}}{1}\over 16}\mid + \mid {\coefAbsGen{\Delta_{p-1}}{1} - \coefAbsGen{\mathJ}{j}\over 16}\mid \\
&\geq \mid {\coefAbsGen{\mathI}{i} - \coefAbsGen{\mathJ}{j}\over 16}\mid.
\end{split}
\end{equation*}
\end{proof}

\section{Collections of Dyadic Cubes and the First Key Lemma}\label{sec: chains}

\begin{definition}\label{def: chain}
Let $\mathJ\subset\mathI$. The \textit{chain} $\chain{\mathI}{\mathJ}$ is the set of dyadic cubes $\mathK$ such that $\mathJ\subseteq\mathK\subseteq\mathI$.
\end{definition}

\begin{definition}\label{def: gen_chain}
A finite set $\mathR\subset\D$ is called a \textit{generalized chain}  if there exists $\mathI_{max}\in\mathR$ such that for any $\mathJ\in\mathR$ one has
\begin{itemize}
\item[a)]\ $\mathJ\subseteq\mathI_{max}$,
\item[b)]\ $\chain{\mathI_{max}}{\mathJ}\subseteq\mathR$.

\end{itemize} The cube $\mathI_{max}\in\mathR$  is called the \textit{maximal cube} of $\mathR$. If $\mathI_{max}\neq [0,1)^d$ then the smallest cube which strictly contains $\mathI_{max}$ is called the  \textit{father of  $\mathR$}, denoted $F(\mathR)$.
\end{definition}

The following lemma is the analogue of  \cite[Lemma~4]{G}. We refer the reader to \cite{G} for the proof.
\begin{lemma}\label{lem: genChain_union}
The union of  two generalized chains $\mathR_1$ and $\mathR_2$ is a generalized chain if and only if either $\mathR_1\cap\mathR_2\neq\emptyset$ or $\father{\mathR_1}\in\mathR_2$ or $\father{\mathR_2}\in\mathR_1$. Then either $\father{\mathR_1\cup\mathR_2} = \father{\mathR_1}$ or $\father{\mathR_1\cup\mathR_2} = \father{\mathR_2}$. 
\end{lemma}

Let us recall two more definitions from \cite{G}.
\begin{definition} Let $\mathS$ be a finite subset of $\D$.  We  say that $\{\mathR_1,\ldots,\mathR_k\}$ is the \textit{minimal generalized chain representation} (MGCR) of $\mathP$ if
\begin{itemize}
\item[a)]\ $\mathS = \bigcup_{i=1}^k \mathR_i$,
\item[b)]\ $\mathR_i$ is a generalized chain for all $1\leq i\leq k$,
\item[c)]\ $\mathR_i\cup\mathR_j$ is not a generalized chain for any $1\leq i<j\leq k$.
\end{itemize}
\end{definition}

It is shown in \cite{G} that any set $\mathS$ has a unique  MGCR.

\begin{definition}
Let $\mathS$ be a finite subset of $\D$ and $\mathI\in\mathS$. We  say that $\mathJ\in\mathS$ is a \textit{son} of $\mathI$ with respect to $\mathS$ if $\chain{\mathI}{\mathJ}\cap\mathS = \{\mathI,\mathJ\}$.
\end{definition}

The set of all sons of $\mathI$ with respect to $\mathS$ will be denoted by $\son{\mathI}{\mathS}$. For any $\mathP\subset\mathS$, put
\begin{equation}
\son{\mathP}{\mathS} :=\bigcup_{\mathI\in\mathP} \son{\mathI}{\mathS},
\end{equation}
and
\begin{equation}
\sonOrder{\mathP}{\mathS}{k+1} :=\son{\sonOrder{\mathP}{\mathS}{k}}{\mathS}.
\end{equation}

Finally, let us define the sets
\begin{itemize}
\item $\Lambda_0(\mathS) = \{\mathI\in\mathS\ :\ \son{\mathI}{\mathS}=\emptyset\}$,
\item $\Lambda_1(\mathS) = \{\mathI\in\mathS\ :\ \sizeOfSet{\son{\mathI}{\mathS}}=1\}$,
\item $\Lambda_2(\mathS) = \{\mathI\in\mathS\ :\ \sizeOfSet{\son{\mathI}{\mathS}}\geq 2\}$.
\end{itemize}
By induction on the cardinality of $\mathS$ (see \cite[p. 56]{G}), one has
\begin{equation}\label{eq: compareLambdas}
\sizeOfSet{\Lambda_2(\mathS)}<\sizeOfSet{\Lambda_0(\mathS)}.
\end{equation}

Now we are ready to formulate and prove the first key lemma.
\begin{lemma}  \label{lem: from_q_to_qtilda}
Let $0<t<s<1$, and let $p,q\in\Ld$ be given, with $\spectrum{p}$  finite and $[0,1]^d \notin \SPectrum{p}$.  Let $\{\mathR_1,\ldots\mathR_k\}$ be the MGCR of $\SPectrum{p}$. Suppose
\begin{itemize}
\item[1)] $\spectrum{p}\cap \spectrum{q}=\emptyset$, 
\item[2)]\ for any $\Delta\in\SPectrum{p}$ there exists $1\leq i\leq 2^d-1$ such that $\coefAbs{\Delta}{i}{p}\geq s$,
\item[3)]\ if $\Delta\in \SPectrum{p}\cap\SPectrum{q}$ then $\coefAbs{\Delta}{j}{q} < {t\over s}\coefAbs{\Delta}{i}{p}$ for some $(\Delta,i)\in\spectrum{p}$ and for any $(\Delta,j)\in\spectrum{q}$,
\item[4)]\ for every $1\leq l\leq k$ there exists $\Delta\in\mathR_l$ and $1\leq i \leq 2^d-1$ such that $\coefAbs{\father{\mathR_l}}{j}{q} < s\coefAbs{\Delta}{i}{p}$ for all $1\leq j\leq 2^d-1$. 
\end{itemize}

Then
\begin{equation}\label{eq: normPplusQ}
\Vert p+q\Vert > C(s,t)\sizeOfSet{\mathM},
\end{equation}
where 
\begin{equation*}
\mathM = \Biggl(\bigcup_{l=1}^k \{ \father{\mathR_l}\}\Biggr)\cup \Bigl(\SPectrum{p}\cap\SPectrum{q}\Bigr).
\end{equation*}
Note that  $[0,1)^d\notin \SPectrum{p}$ by assumption, and so $\father{\mathR_l}$ exists for each $1 \le l \le k$.
\end{lemma}

\begin{proof}
We will proof the lemma for  $d=2$ (the extension to $d \ge 3$ is routine but the notation is more cumbersome). We consider two cases.

CASE 1. $\sizeOfSet{\Lambda_0(\mathM)}\geq {\sizeOfSet{\mathM}\over 12}$. Let $\mathI\in\Lambda_0(\mathM)$.  Then  either  a) $\mathI\in\SPectrum{p}$ or b) $\mathI=\father{\mathR_l}$ for some $1\leq l\leq k$.

a) If $\mathI\in\SPectrum{p}$ then combining 
Lemma \ref{lem: estimate_byCoef} and  condition 2) of this lemma, we conclude that
\begin{equation}\label{eq: normForDisjoint}
\Vert p+q\Vert_{\mathI}\geq s.
\end{equation}

b) If $\mathI = \father{\mathR_l}$ then  the maximal cube of $\mathR_l$  belongs to $\SPectrum{p}$ and therefore satisfies  condition 2) of the Lemma. By  the same argument as in the case a) we conclude that \eqref{eq: normForDisjoint} holds for  $\mathI$. 

Note that  the cubes $\mathI\in\Lambda_0(\mathM)$ are disjoint, so
\begin{equation*}
\Vert p+q\Vert\geq {s\over 12} |\mathM|.
\end{equation*}
CASE 2. $\sizeOfSet{\Lambda_0(\mathM)} < {\sizeOfSet{\mathM}\over 12}$. Taking into account \eqref{eq: compareLambdas} and the fact that $\Lambda_0(\mathM)$ is always nonempty, one has
\begin{equation}\label{eq: Lambda1_isBig}
\sizeOfSet{\Lambda_1(\mathM)} > {5\sizeOfSet{\mathM}\over 6}+1\ \hbox{and}\ \sizeOfSet{\mathM}>12.
\end{equation}
For $\mathI \in \Lambda_1(\mathM)$, let $\mathJ$ denote the unique son of $\mathI$. Let us define
\begin{equation}\label{eq: Lambda1_tilda}
\tilde{\Lambda}_1(\mathM) =\{\mathI\in\Lambda_1(\mathM)\ :\ \mathJ \in\Lambda_1(\mathM)\}.
\end{equation}
Note that
\begin{equation} \label{eq: Lambda1_TildaIsBig}\begin{split}
|\tilde\Lambda_1(\mathM)| &\ge |\Lambda_1(\mathM)| - |\Lambda_0(\mathM)| - |\Lambda_2(\mathM)|\\
& > (\frac{5 |\mathM|}{6} + 1) - \frac{|\mathM|}{12} - \frac{|\mathM|}{12}\\
&= \frac{2 |\mathM|}{3} + 1. \end{split} \end{equation}

For $\mathI\in\tilde{\Lambda}_1(\mathM)$, whose unique son (we recall) is denoted by $\mathJ$, let $\mathK$ denote the unique son of $\mathJ$. Let us prove that either 

\begin{equation}\label{eq: P_plus_Q_on_IJ}
\Vert p+q\Vert_{\mathI\setminus\mathJ} \geq C_1(s,t)
\end{equation} 
or
\begin{equation}\label{eq: P_plus_Q_on_JK}
\Vert p+q\Vert_{\mathJ\setminus\mathK} \geq C_1(s,t).
\end{equation}

If $\mathI\in\SPectrum{p}\cap\SPectrum{q}$ then by conditions  2) and 3) of the lemma, we have $\coefAbs{\mathI}{i}{p}\geq s$ and $\coefAbs{\mathI}{j}{q}< \frac{t}{s}\coefAbs{\mathI}{i}{p}$ 
for some $i$ and $j$. Therefore, according to Lemma \ref{lem: coefDiff_general}, we have
\begin{equation*}
\Vert p+q\Vert_{\mathI\setminus\mathJ} \geq {s-t\over 16}.
\end{equation*}

Using the same argument for the case when  $\mathJ\in\SPectrum{p}\cap\SPectrum{q}$, we have
\begin{equation*}
\Vert p+q\Vert_{\mathJ\setminus\mathK} \geq {s-t\over 16}.
\end{equation*}
It remains to consider the case when $\mathI$, $\mathJ\notin\SPectrum{p}\cap\SPectrum{q}$. Then we have
\begin{equation*}
\mathI = \father{\mathR_{l_1}}\ \hbox{and}\ \mathJ = \father{\mathR_{l_2}},
\end{equation*}
for some $1\leq l_1,l_2\leq k$. Let $\Delta$ and $i$ be chosen according to conditions 2) and  4) of the Lemma for the generalized chain $\mathR_{l_1}$. Then we have that
\begin{equation}\label{eq: existsWithCoef1}
\coefAbs{\Delta}{i}{p}\geq s, 
\end{equation}
and
\begin{equation}\label{eq: existsFatherSmall}
\coefAbs{\mathI}{1}{q}< s\coefAbs{\Delta}{i}{p}. 
\end{equation}

By Lemma \ref{lem: genChain_union} $\mathR_{l_1}$ and $\mathR_{l_2}$ are disjoint, so  $\mathJ\notin\mathR_{l_1}$ and hence $\Delta\not\subseteq\mathJ$ by the definition of  generalized chain. In the case when $\Delta\cap\mathJ = \emptyset$ we get 
\begin{equation*}
\Vert p+q\Vert_{\mathI\setminus\mathJ} \geq s
\end{equation*}
directly from \eqref{eq: existsWithCoef1} and Lemma \ref{lem: estimate_byCoef}. Finally, for the case $\mathJ\subsetneq\Delta$ we apply \eqref{eq: existsWithCoef1}, \eqref{eq: existsFatherSmall} and Lemma \ref{lem: coefDiff_general} to conclude that 
\begin{equation*}
\Vert p+q\Vert_{\mathI\setminus\mathJ}\geq {s(1-s)\over 16}.
\end{equation*}

So, for every $\mathI\in\tilde\Lambda_1(\mathM)$,  either \eqref{eq: P_plus_Q_on_IJ} or \eqref{eq: P_plus_Q_on_JK} holds. Recall that  $\mathJ$ and $\mathK$ depend on $\mathI$. It is to easily seen  that  the sets $\mathI\setminus\mathJ$ as $\mathI$ ranges over $\tilde \Lambda_1(\mathM)$ are disjoint. The same is true for  the sets $\mathJ\setminus\mathK$ as  $\mathI$ ranges over $\tilde \Lambda_1(\mathM)$. Therefore
\begin{equation*}
\Vert p+q\Vert \geq \sum_{\mathI\in\tilde\Lambda_1(\mathM)}  \Vert p+q\Vert_{\mathI\setminus\mathJ},
\end{equation*}
and
\begin{equation*}
\Vert p+q\Vert \geq \sum_{\mathI\in\tilde\Lambda_1(\mathM)}  \Vert p+q\Vert_{\mathJ\setminus\mathK}.
\end{equation*}
Using \eqref{eq: Lambda1_TildaIsBig} we conclude that
\begin{equation*}
\Vert p+q\Vert\geq {1\over 2}\sum_{\mathI\in\tilde\Lambda_1(\mathM)}\Bigl( \Vert p+q\Vert_{\mathI\setminus\mathJ} +  \Vert p+q\Vert_{\mathJ\setminus\mathK}\Bigr)\geq C(s,t)\sizeOfSet{\mathM}.
\end{equation*}

\end{proof}

\section{Symmetrization Properties}\label{sec: symmetry}

We will prove the results in this section only for $d=2$, but the extension to $d \ge 3$ is routine. 

 Let $\Delta=[a,a+2\delta)\times [b,b+2\delta)$ be a dyadic square of side length $2\delta$, and let $\Delta_i$, $1\leq i\leq 4$, be the four disjoint 
immediate successor squares of  $\Delta$  of side length $\delta$ . For each $1 \le i \le 4$, let us denote by $L_i(f,\Delta)$ the function which agrees with $f$ on the sets $[0,1)^2\setminus\Delta$ and  $\Delta_i$, 
and which

 `copies' $f$ from the square $\Delta_i$ to the other three squares $\Delta_j,\ j\neq i$. More precisely, let $\Delta_j = u_j + \Delta_i$, where $u_j \in \mathbb{R}^2$. Then 
$$ L_i(f, \Delta)(x) = f(x - u_j) \qquad\text{for $1 \le j \le 4$ and for all $x \in \Delta_j$}.$$ 

\begin{lemma}\label{lem: basic_symmetry}

 let $f,g\in\LtwoDim$,  let $\SPectrum{f+g}\neq\emptyset$, and let
\begin{equation}\label{eq: def_constantB}
B := {\Vert f\Vert\over \Vert f+g\Vert}.
\end{equation}

Then (for any square $\Delta$) we have either

\begin{equation}\label{eq: symmetry_greatForSomeI}
\Vert L_i(f,\Delta)\Vert > B\Vert L_i(f+g,\Delta)\Vert,\ \hbox{for some}\ 1\leq i\leq 4,
\end{equation}
or
\begin{equation}\label{eq: symmetry_equalForAllI}
\Vert L_i(f,\Delta)\Vert = B\Vert L_i(f+g,\Delta)\Vert,\ \hbox{for all}\ 1\leq i\leq 4.
\end{equation}
\end{lemma}
\begin{proof}
Let us assume that the statement of the lemma is not correct. Then we have 
\begin{equation}\label{eq: symmetry_smallForSum}
\sum_{i=1}^4\Vert L_i(f,\Delta)\Vert < B\sum_{i=1}^4\Vert L_i(f+g,\Delta)\Vert.
\end{equation}

Note that
\begin{equation*}
\Vert L_i(f,\Delta)\Vert = \Vert f\Vert + 4\Vert f\Vert_{\Delta_i} - \sum_{j=1}^4\Vert f\Vert_{\Delta_j}
\end{equation*}
and
\begin{equation*}
\Vert L_i(f+g,\Delta)\Vert = \Vert f+g\Vert + 4\Vert f+g\Vert_{\Delta_i} - \sum_{j=1}^4\Vert f+g\Vert_{\Delta_j}.
\end{equation*}
By substituting the last two inequalities into \eqref{eq: symmetry_smallForSum} we conclude that,
\begin{equation*}
4\Vert f\Vert < 4B\Vert f+g\Vert,
\end{equation*}
which is contradiction. 
\end{proof}
\begin{remark}
If $f \neq 0$ then $L_j(f,\Delta)\neq 0$  for some $1\leq j\leq 4$.
\end{remark}

For $f,g\in\LtwoDim$, with $f+g\neq 0$, and $\Delta\in\D^2$, define
\begin{equation}\label{eq: operator_L}
L((f,g),\Delta):=\begin{cases} (L_i(f,\Delta),L_i(g,\Delta)),\ \hbox{if}\ \eqref{eq: symmetry_greatForSomeI}\ \hbox{holds and}\ \Vert L_i(f+g,\Delta)\Vert>0\\
(L_j(f,\Delta),L_j(g,\Delta)), \hbox{where } \Vert L_j(f+g,\Delta)\Vert>0\hbox{ otherwise}.
\end{cases}
\end{equation}

\begin{lemma}\label{lem: symmetry_main}
Let $f,g\in\LtwoDim$ and $\Delta\in\D^2$ satisfy

\begin{itemize}
\item[i)]\ $\Delta\notin\SPectrum{f}$,\ $\Delta\notin\SPectrum{g}$,
\item[ii)] $\spectrum{f}\cap\spectrum{g}=\emptyset$,
\item[iii)] $\spectrum{f+g}\neq\emptyset$.
\end{itemize}\end{lemma}
Further, let $(f',g'):=L((f,g),\Delta)=L_i((f,g),\Delta)$ for some $1 \le i \le 4$. Then,
\begin{itemize}
\item[1)]\ $\Delta\notin\SPectrum{f'}$,\ $\Delta\notin\SPectrum{g'}$,
\item[2)] $\coefAt{\mathI}{j}{f'}=\coefAt{\mathI}{j}{f}$ and $\coefAt{\mathI}{j}{g'}=\coefAt{\mathI}{j}{g}$ for all $\mathI$ with $\mathI\not\subset\Delta\setminus\Delta_i$,
\item[3)] for every $\mathI\subset\Delta\setminus\Delta_i$ there exists $\mathJ\in\Delta_i$ such that 
$\coefAt{\mathI}{j}{f'}=\coefAt{\mathJ}{j}{f}$ and $\coefAt{\mathI}{j}{g'}=\coefAt{\mathJ}{j}{g}$ for all $1\leq j\leq 4$,
\item[4)] $\spectrum{f'}\cap\spectrum{g'}=\emptyset$,
\item[5)] $\spectrum{f'+g'}\neq\emptyset$,
\item[6)]\ ${\Vert f'\Vert\over\Vert f'+g'\Vert}\geq{\Vert f\Vert\over\Vert f+g\Vert}$.
\end{itemize}
\begin{proof} From condition i) of the lemma and  \eqref{eq: haarcoef} it follows that 
\begin{equation} \label{eq: ints}
\int_{\Delta_j}f = \int_{\Delta_k}f\ \hbox{and}\ \int_{\Delta_j}g = \int_{\Delta_k}g\ \hbox{for all}\ 1\leq j,k\leq 4. 
\end{equation}
The functions $f'$ and $g'$ (replacing $f$ and $g$) will also satisfy \eqref{eq: ints}, which gives statement 1). Clearly,  $\int_{\Delta_i}f = \int_{\Delta_i}f'$ and $\int_{\Delta_i}g = \int_{\Delta_i}g'$, and hence
\begin{equation}\label{eq: integralEqual}
\int_\Delta f = \int_\Delta f'\quad \hbox{and}\quad \int_\Delta g = \int_\Delta g'.
\end{equation}  
If $\mathI\not\subset\Delta\setminus\Delta_i$ then either $\mathI\cap\biggl(\Delta\setminus\Delta_i\biggr)=\emptyset$ or $\mathI=\Delta$ or $\mathI\supsetneq\Delta$. In the first case we have $f=f'$ and $g=g'$ on $\mathI$ and therefore we have statement 2). The second case is equivalent to  statement 1). In the last case it is easy to check that $\coefAt{\mathI}{j}{f'}$ and $\coefAt{\mathI}{j}{g'}$ depend only on the average  values $\int_\Delta f'$ and $\int_\Delta g'$ on $\Delta$, which implies, taking \eqref{eq: integralEqual} into account,   that $\coefAt{\mathI}{j}{f'}=\coefAt{\mathI}{j}{f}$ and $\coefAt{\mathI}{j}{g'}=\coefAt{\mathI}{j}{g}$. Thus, the proof of statement 2) is complete.
Statement 3) is obvious and follows from the definition of the operator $L$.

Statement 4) follows from the condition ii) and statements 2) and 3). Statement 5) follows immediately from the definition of $(L(f,g),\Delta)$.

Finally, let us prove  statement 6). The statement is obvious if $f'$ and $g'$ are defined according to the first line of \eqref{eq: operator_L}. Let us consider the case when they are defined according to the second line. We have two cases.

CASE 1. Suppose \eqref{eq: symmetry_greatForSomeI} fails. Then, by  Lemma~\ref{lem: basic_symmetry},  the set of equalities \eqref{eq: symmetry_equalForAllI} holds, which implies that we have equality in  statement 6).

CASE 2. Now suppose \eqref{eq: symmetry_greatForSomeI} holds, but that $L_i(f+g,\Delta)=0$. Since $\spectrum{L_i(f)}\cap\spectrum{L_i(g)}=\emptyset$ we  conclude that $L_i(f)=L_i(g)=0$. So  inequality in  \eqref{eq: symmetry_greatForSomeI} is impossible, which contradicts our assumption.
\end{proof}
\section{The Second Key Lemma}

\begin{lemma}\label{lem: symmetry_main1}
Let $f, g\in\Ld$, with $\spectrum{f}$ finite, and let $\{\mathR_1,\ldots,\mathR_k\}$ be the MGCR of $\SPectrum{f}$.  Suppose that
\begin{itemize}
\item[1)]\ $\SPectrum{f}\cap\SPectrum{g}=\emptyset$,
\item[2)]\ $[0,1)^d\notin \SPectrum{f}$ and $[0,1)^d\notin \SPectrum{g}$,
\item[3)]\ $\father{\mathR_l}\notin\SPectrum{g}$ for all $1\leq l\leq k$,
\item[4)]\ $\mid\coefAt{\mathI}{j}{g}\mid\leq 1$ for all $\mathI\in\D$ and $1\leq j\leq 2^d-1$,
\item[5)]\ for every $1\leq l\leq k$ there exist $\mathI\in\mathR_l$ and $1\leq j\leq 2^d-1$ such that $\mid\coefAt{\mathI}{j}{f}\mid\geq t$.
\end{itemize}
Then
\begin{equation}
{\Vert f\Vert\over \Vert f+g\Vert}\leq C(t).
\end{equation}
\end{lemma}

\begin{proof}  
We may assume that $\mu({\father{\mathR_1}})\leq\mu({\father{\mathR_2}})\leq\ldots\leq\mu({\father{\mathR_k}})$. Let
\begin{equation*}
(f_1,g_1) := L(f,g,\father{\mathR_1}).
\end{equation*}

Let  $V_1\subset\father{\mathR_1}$ denote the cube from which the values of  $f$ and $g$ are copied to  the other immediate successor cubes of $\father{\mathR_1}$ to define $L(f,g)$. Note that $\SPectrum{f_1}$ is obtained as follows:
\begin{itemize}
\item[1)] `Remove' all generalized chains from $MGCR(f)$ whose maximal  cubes are contained in $\father{\mathR_1}\setminus V_1$,
\item[2)] `Copy' all generalized chains  from $MGCR(f)$ whose maximal  cubes are contained in $V_1$ to  the other $2^d-1$  immediate successor cubes of $\father{\mathR_1}$.
\end{itemize} 
 Then $f_1$ and $g_1$ have all properties that are listed in Lemma~ \ref{lem: symmetry_main}. In particular,
\begin{equation*}
{\Vert f\Vert\over\Vert f+g\Vert}\leq{\Vert f_1\Vert\over\Vert f_1+g_1\Vert}.
\end{equation*}
Inductively, for $i=2,\ldots,k$, let
\begin{equation*}
(f_i,g_i) := L(f_{i-1},g_{i-1},\father{\mathR_i}).
\end{equation*}
Finally, let $f':=f_k$ and $g':=g_k$. It is easy to check that $f'$ and $g'$ have the following properties:
\begin{itemize}
\item[P1)]\ $\SPectrum{f'}\cap\SPectrum{g'}=\emptyset$,
\item[P2)]\ $[0,1)^d\notin \SPectrum{f'}$ and $[0,1)^d\notin \SPectrum{g'}$,
\item[P3)]\ $\father{\mathR^{'}_l}\notin\SPectrum{g'}$ for all $1\leq l\leq k^{'}$,
\item[P4)]\ $\mid\coefAt{\mathI}{j}{g'}\mid\leq 1$ for all $\mathI\in\D$ and $1\leq j\leq 2^d-1$,
\item[P5)]\ for every $1\leq l\leq k^{'}$ there exists $\mathI\in\mathR^{'}_l$ and $1\leq j\leq 2^d-1$ such that $\mid\coefAt{\mathI}{j}{f^{'}}\mid\geq t$,
\item[P6)]\ for every $1\leq l\leq k^{'}$ the functions $f'$ and $g'$ are `copied' from one immediate successor cube of the cube $\father{\mathR^{'}_l}$ to all the other $2^d - 1$ immediate successors,
\item[P7)]\ ${\Vert f\Vert\over\Vert f+g}\Vert\leq {\Vert f'\Vert\over\Vert f'+g'\Vert}$,
\end{itemize}
where $\{\mathR^{'}_1,\ldots,\mathR^{'}_{k^{'}}\}$ is the MGCR of $\SPectrum{f'}$. Let
$$\mathS := \{\father{\mathR^{'}_l}\ :\ 1\leq l\leq k^{'}\}.$$

Suppose that $\coefAt{\mathI}{j}{f'}\neq 0$. Then $\mathI$ belongs to some
generalized chain $\mathR^{'}_l$ and therefore there exists
$\mathJ\in\mathS$ such that $\mathI\subset\mathJ$. So we conclude
that
\begin{equation}\label{eq: fPrime_support}
supp(f')\subseteq \bigcup_{\mathI\in
\mathS} \mathI.
\end{equation}

We  say that $\mathI\in\mathS$ has \textit{order $k$} if
\begin{equation*}
\sonOrder{\mathI}{\mathS}{k}\neq\emptyset\quad\hbox{and}\quad
\sonOrder{\mathI}{\mathS}{k+1}=\emptyset.
\end{equation*}

Let us prove that for any $\mathI\in\mathS$
\begin{equation}\label{eq: statement_induction}
\Vert f'\Vert_\mathI<\bigl(5t^{-1}+2\bigr)\Vert f'+g'\Vert_\mathI-2t-8.
\end{equation}
We use induction on the order of $\mathI$. Suppose the order of $\mathI$ is $0$. Let $\mathI_1, \mathI_2,\ldots,\mathI_{2^d}$ be the immediate successor cubes of $\mathI$. Since $\mathI\in\mathS$ then, taking into account P5), P6) and Lemma \ref{lem: estimate_byCoef}, one has
\begin{equation*}
\Vert f'+g'\Vert_{\mathI_j} \geq t\ \hbox{for all}\ 1\leq j\leq 2^d.
\end{equation*}
Hence
\begin{equation}\label{eq: finalLemma_Norm1}
\Vert f'+g'\Vert_{\mathI} \geq 2^d t.
\end{equation}
It follows from Lemma \ref{lem: estimate_forQ} and P4) that
\begin{equation}\label{eq: finalLemma_Norm2}
\Vert \mathP_\mathI(g')\Vert_\mathI\leq 1.
\end{equation}
Since the order of $\mathI$ is equal to $0$, the monotonicity of the Haar system and P1) give
\begin{equation}\label{eq: haarmonotorderzero}
\Vert f'+g'\Vert_\mathI\geq \Vert
f'+\mathP_\mathI(g')\Vert_\mathI.
\end{equation}
Combining  \eqref{eq: finalLemma_Norm1}, \eqref{eq: finalLemma_Norm2} and \eqref{eq: haarmonotorderzero} gives
\begin{equation}
\begin{split}
(5t^{-1}+2\bigr)\Vert f'+g'\Vert_\mathI&\geq (5t^{-1}+1\bigr)\Vert f'+g'\Vert_\mathI+\Vert
f'+\mathP_\mathI(g')\Vert_\mathI\\ 
&\geq 2^d t(5t^{-1}+1\bigr)
+ \Vert f'+\mathP_\mathI(g')\Vert_\mathI  \\ 
&\geq
10+2t + \Vert f'+\mathP_\mathI(g')\Vert_\mathI  \\ &> \Vert
f'\Vert_\mathI+2t+9,
\end{split}
\end{equation} which gives
\eqref{eq: statement_induction} for $\mathI$.

Assume now that \eqref{eq: statement_induction} holds for all cubes of order $\leq k$.
We will prove the estimate for  all cubes $\mathI\in\mathS$ of
order $k+1$. Let $\mathI_1,\ldots,\mathI_{2^d}$ be the immediate successor cubes of
$\mathI$ and let $\son{\mathI}{\mathS} = \{\mathJ_1,\mathJ_2,\dots,\mathJ_{i}\}$. Because of the symmetry property  P6), each cube $\mathI_r$ contains the same number of the  $\mathJ_p$ cubes, $a$, say, where $a \ge 1$. Let $\mathJ_1,\dots,\mathJ_a$ be contained in $\mathI_1$.  The cubes $\mathJ_i$ are disjoint and their orders are $\leq k$.
Therefore, by the induction hypothesis,
\begin{equation}   \label{eq: finalLemma_NormOnIi}
\Vert f'\Vert_{\mathJ_i}\leq (5t^{-1}+2\bigr)\Vert
f'+g'\Vert_{\mathJ_i}-2t-8\quad\hbox{for all}\ i=1,2,\dots, a.
\end{equation}
Let $D :=\mathI_1\setminus\biggl(\bigcup_{i=1}^a
\mathJ_i\biggr)$,
\begin{equation*}
\alpha := \mathP_{\mathJ_1}(\mathP_{\mathJ_2}(\dots
\mathP_{\mathJ_a}(f')\dots)),
\end{equation*}
and
\begin{equation*}
\beta := \mathP_{\mathJ_1}(\mathP_{\mathJ_2}(\dots
\mathP_{\mathJ_s}(g')\dots)).
\end{equation*}
For $\alpha$ and $\beta$, we have
\begin{equation}\label{eq: finalLemma_AlphaBeta}
\alpha= f'\quad\hbox{and}\quad \beta=g'\ \hbox{on}\ D, 
\end{equation}
and by Lemma \ref{lem: estimate_forQ} and P4)
\begin{equation}\label{eq: finalLemma_normBeta}
\Vert \beta\Vert_{\mathJ_i}\leq 1\quad \hbox{for any}\  1\leq i\leq s.
\end{equation}
By Lemma~ \ref{lem: estimate_forQ}  and the monotonicity of the Haar system, we have
\begin{equation}\label{eq: finalLemma_normAlphaPlusBeta}
\Vert\alpha+\beta\Vert_{\mathI_1}\geq \Vert \alpha+\mathP_{\mathI_1}(\beta)\Vert_{\mathI_1}\geq \Vert
\alpha\Vert_{\mathI_1}-1.
\end{equation}
Using \eqref{eq: finalLemma_normBeta}, we get
\begin{equation*}
\begin{split}
\Vert \alpha+\beta\Vert_{\mathI_1\setminus D}&\leq \Vert \alpha\Vert_{\mathI_1\setminus
D}+\Vert \beta\Vert_{\mathI_1\setminus D}\\&=
\Vert \alpha\Vert_{\mathI_1\setminus D}+ \sum_{i=1}^a
\Vert\beta\Vert_{\mathJ_i}\\&\leq\Vert
\alpha\Vert_{\mathI_1\setminus D}+a.
\end{split}
\end{equation*}
From \eqref{eq: finalLemma_AlphaBeta} and \eqref{eq: finalLemma_normAlphaPlusBeta}, we get
\begin{equation*}
\begin{split}
\Vert f'+g'\Vert_D &= \Vert \alpha+\beta\Vert_D\\ &= \Vert
\alpha+\beta\Vert_{\mathI_1} - \Vert
\alpha+\beta\Vert_{\mathI_1\setminus D}\\&\geq
\Vert\alpha\Vert_{\mathI_1} -1 - \Vert
\alpha\Vert_{\mathI_{1}\setminus D}-a\\&= \Vert \alpha\Vert_{D}-a-1\\ &=
\Vert f'\Vert_{D}-a-1.
\end{split}
\end{equation*}
Combining this with  \eqref{eq: finalLemma_NormOnIi} gives 
\begin{equation*}
\begin{split}
\Vert f'\Vert_{\mathI_1} &= \Vert f'\Vert_{D} + \sum_{i=1}^a
\Vert f'\Vert_{\mathJ_i}\\&\leq 
\Vert f'+g'\Vert_{D}+a+1 +(5t^{-1}+2)\sum_{i=1}^a \Vert
f'+g'\Vert_{\mathJ_i}-a(2t+8)\\&\leq
(5t^{-1}+2)\Vert f'+g'\Vert_{\mathI_1}-2t-6.
\end{split}
\end{equation*}
Using the symmetry property P6), we conclude
\begin{equation*} \begin{split}
\Vert f'\Vert_{\mathI} &= 2^d\Vert f'\Vert_{\mathI_1}\\ &\leq (5t^{-1}+2)\Vert
f'+g'\Vert_{\mathI} - 2^d(2t+6)\\ &< (5t^{-1}+2)\Vert f'+g'\Vert_{\mathI}
- 2t-8. \end{split}
\end{equation*}
So \eqref{eq: statement_induction} holds for $\mathI$. This completes the induction proof.

Let $B := \bigcup_{\mathI\in\mathS}\mathI$. We can represent the
set $B$ as the union of some disjoint intervals from $\mathS$. Hence
\begin{equation*}  
\Vert f'\Vert_B < (5t^{-1}+2)\Vert f'+g'\Vert_B.
\end{equation*}
By \eqref{eq: fPrime_support}  we have $f'=0$ on the set $[0,1]^d\setminus B$, so
\begin{equation*} 
\Vert f'\Vert < (5t^{-1}+2)\Vert f'+g'\Vert.
\end{equation*}
This, together with P7), completes the proof.
\end{proof} \section{Main Results}
\begin{proof}[Proof of the Main Theorem]
Convergence is obvious if $f-G_m^{s,t}(f)=0$. So assume that $G_m^{(s,t)}(f)\neq f$. Let us define
\begin{equation}\label{eq: defintePandQ}
\begin{split}
p:=&{G_m^{(s,t)}(f)\over \max{\mid \coefAt{\Delta}{j}{f-G_m^{(s,t)}(f)}\ :\ (\Delta,j)\in\spectrum{f}\mid}},\\
q:=&{f-G_m^{(s,t)}(f)\over \max{\mid \coefAt{\Delta}{j}{f-G_m^{(s,t)}(f)}\ :\ (\Delta,j)\in\spectrum{f}\mid}}.
\end{split}
\end{equation}
It is clear that  $p$ and $q$ satisfy all the conditions of  Lemma \ref{lem: from_q_to_qtilda}. Therefore we have the estimate \eqref{eq: normPplusQ}. Define
\begin{equation}
\tilde{q} := q - \sum_{\mathI\in\mathM}\sum_{j=1}^{2^d-1}\coefAt{\mathI}{j}{q}\haar{\mathI}{j}.
\end{equation}
Clearly,
\begin{equation*}
\Vert q-\tilde{q}\Vert \leq(2^d-1)\ |\mathM|.
\end{equation*}
Combining this with  inequality \eqref{eq: normPplusQ} yields
\begin{equation}\label{eq: normPplusQtilde}
{\Vert p+\tilde{q}\Vert \over\Vert p+q\Vert}\leq 1+{\Vert q-\tilde{q}\Vert \over\Vert p+q\Vert}\leq 1+{2^d-1\over C(s,t)}.
\end{equation}
It remains to observe that functions $p$ and $\tilde{q}$ (in place of $f$ and $g$, respectively) satisfy  the conditions of  Lemma~ \ref{lem: symmetry_main}. Hence
\begin{equation*}
{\Vert p\Vert \over\Vert p+\tilde{q}\Vert}\leq C(t).
\end{equation*}
Combining this with   \eqref{eq: normPplusQtilde} and \eqref{eq: defintePandQ} we get
\begin{equation*} 
{\Vert G_m^{(s,t)}(f)\Vert\over\Vert f\Vert} \leq C(t)\cdot\Bigl(1+{2^d-1\over C(s,t)}\Bigr),
\end{equation*} which gives the uniform boundedness inequality \eqref{eq: uniformbounded} of the Main Theorem
with \begin{equation} \label{eq: greedyapproximants} C(s,t,d) = C(t)\cdot\Bigl(1+{2^d-1\over C(s,t)}\Bigr). \end{equation}
To deduce convergence of the algorithm from the uniform  boundedness of the greedy approximants, we follow the argument from \cite{W} for the Thresholding Greedy Algorithm.  Let $\Psi:= (\psi_n)_{n=0}^\infty$ be the enumeration of the multivariate Haar basis in the order determined by $\prec$. Then $\Psi$ is a Schauder basis for $L_1[0,1]^d$.  For $N \ge 0$, Let $P_N$ denote the basis projection onto $\operatorname{span} (\psi_j)_{j=0}^N$. Suppose that $N = (2^d-1)N_1+1$ so that $P_N$ projects onto the span of all Haar functions supported on the first $N_1$ dyadic cubes in the order $\prec$.
It is easily seen from the definition of the algorithm that for  all sufficiently large $n$ there exists $m$ such that
$$G^{s,t}_n(f) = P_N(f) + G_m^{s,t}(f-P_N(f)).$$
Since $\|G_m^{s,t}(f - P_N(f))\| \le C(s,t,d)\|f - P_N(f)\| \rightarrow 0$ as $N \rightarrow \infty$, we get $f = \lim_{n \rightarrow \infty} G^{s,t}_n(f)$ as required. \end{proof}
\begin{corollary} For every $0<t<1$ there is a convergent implementation of the weak greedy algorithm with weakness parameter $t$ for the multivariate Haar basis for $\Ld$ such that, for all $f \in \Ld$, we have
$$\|G^t_n(f)\| \le \frac{C_d}{1-t}\|f\|,$$
where $C_d \ll 2^d$.\end{corollary} \begin{proof}  From the proof of Lemma~\ref{lem: from_q_to_qtilda},  $C(s,t) \ge \min(s(1-s),s-t)/24$,
and from the proof of  Lemma~\ref{lem: symmetry_main}, $C(t) \le 5/t + 12$. Setting $s = (1+t)/2$ and substitutuing these estimates into \eqref{eq: greedyapproximants} gives the result for $G_n^t(f) := G_n^{s,t}(f)$.
\end{proof}

Finally, let us show that $C_d \gg \sqrt{d}$. Let $(r_n)_{n=1}^\infty$ be the usual Rademacher functions defined on $[0,1]$. For each finite $A \subset \mathbb{N}$, recall that the Walsh function $w_A := \prod_{n \in A} r_n$. The Walsh system $\mathcal{W} := \{ w_A  \colon A \in \mathbb{N}^{(<\infty)}\}$ is a fundamental orthogonal system for $L_1[0,1]$.

\begin{theorem} Let $0< t < 1$. The greedy approximants with respect to $\mathcal{W}$ are unbounded for every implementation of the weak thresholding greedy algorithm with weakness parameter $t$. 
\end{theorem} \begin{proof} Let $0 < u < t$. For $N \ge 1$, consider 
$$ f_N := \prod_{n=1}^N (1 + ur_n).$$
From the independence of the Rademacher functions and the fact that $1 + ur_n\ge0$, we get $\|f_N\| = 1$.
Also, for \textit{any} implementation of the weak thresholding greedy algorithm with weakness parameter $t$, we have 
$$G^t_{N+1}(f_N) = 1 + u\sum_{n=1}^N r_n.$$ So, by Khinchine's inequality, \begin{equation} \label{eq: Khinchine}
\|G^t_{N+1}(f_N)\| \ge u\|\sum_{n=1}^N r_n\| - 1 \gg u \sqrt{N}, \end{equation}
which gives the unboundedness of the greedy approximants.
\end{proof} \begin{corollary} $C_d \gg \sqrt{d}$. \end{corollary}
\begin{proof} It suffices to observe that  the Haar functions $\chi_{[0,1)^d} \cup (h^{(j)}_{[0,1)^d})_{j=1}^{2^d-1}$ have the same joint distribution as the initial segment of the Walsh system $\{w_A \colon A\subset\{1,\dots,d\}\}$. Hence the result follows from \eqref{eq: Khinchine}.
\end{proof}
\section{The Boundary Cases $s=1$ and $s=t$}
In this section we will assume that $d=2$, that $N=2k$ is even number, and that $0<\epsilon<1$. Let  $\Delta_n := [0,{1\over 2^n})\times [0,{1\over 2^n})$. Our starting point is the function
\begin{equation}
f_N = 1 + \sum_{n=0}^{N-1}\sum_{j=1}^3 \haar{\Delta_n}{j}.
\end{equation}
It is easy to check, that 
\begin{equation}
f_N=\begin{cases}2^{2N}\ \hbox{on}\ \Delta_N,\\ 0\ \hbox{otherwise}.\end{cases}
\end{equation}
So $\Vert f_N\Vert=1$.

Now let us consider the case when $s=1$. Let 
\begin{equation*}
f_N^{\epsilon} := 1 + \sum_{n=0}^{k-1}\sum_{j=1}^3 \bigl(\haar{\Delta_{2n+1}}{j} + (1-\epsilon)\haar{\Delta_{2n}}{j}\bigr). 
\end{equation*}
It is clear that $\Vert f_N^\epsilon\Vert\leq \Vert x_N\Vert + 3k\epsilon$. On the other hand, it is easy to check that
\begin{equation*}
G_{3k+1}^{1,t}(f_N^{\epsilon}) = 1 + \sum_{n=0}^{k-1}\sum_{j=1}^3 \haar{\Delta_{2n+1}}{j}.
\end{equation*}
By using Lemma \ref{lem: coefDiff_Neighboor} for $\mathI=\Delta_{2n+1}$, $\mathJ=\Delta_{2n+2}$, and $\mathK= \Delta_{2n+3}$, we get
 $$\Vert G_{3k+1}^{1,t}(f_N^{\epsilon})\Vert \ge \sum_{n=0}^{k-1} \Vert
G_{3k+1}^{1,t}(f_N^{\epsilon}) \Vert_{\Delta_{2n+1}\setminus \Delta_{2n+3}} \geq \frac{k}{8},$$ so
\begin{equation*}
{\Vert G_{3k+1}^{1,t}(f_N^{\epsilon}) \Vert\over\Vert f_N^\epsilon\Vert} \geq {k\over 8(1+3k\epsilon)}.
\end{equation*}
Since $k$ and $\epsilon$ are arbitrary, we conclude that the operator $G_n^{1,t}$ is not bounded.

Now, let us consider the case $s=t$. Let
\begin{equation}
g_N^\epsilon := t\Biggl(1 + \sum_{n=0}^{N-1} \biggl(\haar{\Delta_n}{1} + \haar{\Delta_n}{2} + (1-\epsilon)\haar{\Delta_n}{3}\biggr)\Biggr) + \haar{\Delta_N}{1}.
\end{equation}  Note that
$$ g_N^\epsilon = t f_N  - t\epsilon \sum_{n=0}^{N-1}\haar{\Delta_n}{3}+ \haar{\Delta_N}{1},$$
whence by the triangle inequality $\Vert g_N^{\epsilon}\Vert \leq 1+t + Nt\epsilon$.  Note that
\begin{equation}
G_{2N+1}^{t,t}(g_N^\epsilon) = t\Biggl(1 + \sum_{n=0}^{N-1} \biggl(\haar{\Delta_n}{1}+ \haar{\Delta_n}{2}\biggr)\Biggr).
\end{equation} Hence by Lemma~\ref{lem: coefDiff_SameLevel}
$$\Vert G_{2N+1}^{t,t}(g_N^\epsilon)\Vert \geq  \sum_{n=0}^{N-1}\Vert  G_{2N+1}^{t,t}(g_N^\epsilon) \Vert_{\Delta_n \setminus \Delta_{n+1}} \ge \frac{Nt}{2}.$$
So $$ \frac{\Vert G_{2N+1}^{t,t}(g_N^\epsilon)\Vert}{\Vert g_N^{\epsilon}\Vert} \ge \frac{Nt}{2(1 + t + Nt\epsilon)}.$$
Since $N$ and $\epsilon$ are arbitrary, we conclude that the operator $G_n^{t,t}$ is not bounded.

 By a ``gliding hump'' argument (see \cite{W} for the details)  there exists $f \in \Ld$ for which $(G^{t,t}_n(f))$ diverges.

The case $s=t=1$ implies the following result.

\begin{corollary} \label{cor: notQG} The multivariate Haar system is not a quasi-greedy basis of $\Ld$.
\end{corollary}

\section{Appendix: Proof of Lemma~\ref{lem: coefDiff_SameLevel}}
Here we prove Lemma \ref{lem: coefDiff_SameLevel} for $d\geq 3$. The first statement is identical to the case $d=2$. We prove the second and third statements. Let $H$ be the value of $\mathS_\mathI(f)$ on $\mathI$ and also $\mu(\mathI)=\delta$. Define signs $\sigma_p=\pm 1$ in such a way that the value of $\mathS_\mathJ(f)$ is equal to 
$H + \frac{1}{\delta}\sum^{2^d-1}_{p=1} a_p$ where $a_p=\sigma_p\coefAt{\mathI}{p}{f}$. Let $(\mathJ_j)_{j=1}^{2^d-1}$ be an enumeration of the 
$2^{d}-1$ immediate successors of $\mathI$ excluding $\mathJ$. For $1 \le j \le 2^d-1$, let $\epsilon_j^{(p)} = \pm1$ denote the value taken by $\delta \sigma_p h_{\mathI}^{(p)}$ on the cube $\mathJ_j$.

 Then by monotonicity
\begin{equation}
\Vert f\Vert_{\mathI\setminus\mathJ}\geq \frac{\delta}{2^d} \sum_{j=1}^{2^d-1}\mid H + {1\over\delta}\sum_{p=1}^{2^d-1}\epsilon_j^{(p)}a_p\mid.
\end{equation}

For a fixed $j$ (resp., for a fixed $p$) there are exactly $2^{d-1}$  of the coefficients $\{\epsilon_j^{(p)}\}$ that are equal to $-1$, and
for distinct $p$ and $q$  the following orthogonality property is easily verified:
\begin{equation} \label{eq: orthogonalitycond}  \sum_{\{j \colon \epsilon^{(p)}_j = -1\}} \epsilon^{(q)}_j = 0. \end{equation} 
 Let us fix some $p_0$. Using \eqref{eq: orthogonalitycond} we get
\begin{equation*} \begin{split}
\Vert f\Vert_{\mathI\setminus\mathJ} &\geq {\delta\over 2^d}\mid \sum_{\{j \colon \epsilon^{(p_0)}_j = -1\}}
\Bigl(H + {1\over\delta}\sum_{p=1}^{2^d-1}\epsilon_j^{(p)}a_p\Bigr)\mid\\ &= {1\over 2^d} \mid 2^{d-1}H\delta - 2^{d-1}a_{p_0}\mid={\mid H\delta-a_{p_0}\mid\over 2},  \end{split}
\end{equation*}
which proves the second statement of the lemma.

 Finally, for distinct $p$ and $q$, we have 
\begin{equation*} \begin{split}
\Vert f\Vert_{\mathI\setminus\mathJ} &\geq {1\over 2} \Bigl( {\mid H\delta - a_{p}\mid\over 2} +{\mid H\delta - a_{q}\mid\over 2}\Bigr)\\ &\geq {\mid a_{p} - a_{q}\mid\over 4}\geq \mid {\mid\coefAt{\mathI}{p}{f}\mid-\mid\coefAt{\mathI}{q}{f}\mid\over 4}\mid, \end{split}
\end{equation*} 
which proves the last statement of the lemma.

\end{document}